\documentclass[review,onefignum,onetabnum]{siamart220329}

\usepackage{lipsum}
\usepackage{amsmath,amsfonts,amssymb}
\usepackage{graphicx}
\usepackage{epstopdf}
\usepackage{algorithmic}
\usepackage{color}
\usepackage{ulem}
\usepackage{comment}
\usepackage{bm}
\usepackage{tikz}
\usepackage{mathtools}
\usepackage{subcaption}
\usepackage{multirow}
\usepackage{threeparttable}
\usepackage{framed}

\usetikzlibrary{intersections,calc,arrows.meta}

\usepackage{booktabs} 

\ifpdf
  \DeclareGraphicsExtensions{.eps,.pdf,.png,.jpg}
\else
  \DeclareGraphicsExtensions{.eps}
\fi

\crefname{ALC@unique}{Step}{Steps}

\newsiamthm{remark}{Remark}
\newsiamthm{example}{Example}
\newsiamthm{problem}{Problem}
\crefname{problem}{Problem}{Problems}
\newsiamthm{assumption}{Assumption}
\crefname{assumption}{Assumption}{Assumptions}

\headers{Infinite-dimensional Controllability Score}{Y.~NAKABE AND K.~SATO}

\title{Extension of Controllability Score to\\ Infinite-Dimensional Systems\thanks{Submitted to the editors DATE.
\funding{This work was supported by Japan Society for the Promotion of Science KAKENHI under 23K03899.}}}

\author{Yuito Nakabe\thanks{Department of Mathematical Informatics, Graduate School of Information Science and Technology, The University of Tokyo (Y.~Nakabe: \email{nakabe-yuito1111@g.ecc.u-tokyo.ac.jp}, K.~Sato: \email{kazuhiro@mist.i.u-tokyo.ac.jp}).}
\and Kazuhiro Sato \footnotemark[2]}

\usepackage{amsopn}
\DeclareMathOperator{\diag}{diag}
\DeclareMathOperator{\Span}{span}
\DeclareMathOperator{\tr}{tr}
\DeclareMathOperator{\dom}{dom}
\DeclareMathOperator{\ran}{ran}

\ifpdf
\hypersetup{
  pdftitle={An Example Article},
  pdfauthor={D. Doe, P. T. Frank, and J. E. Smith}
}
\fi

\begin{document}

\maketitle
\begin{abstract}
Centrality analysis in dynamical network systems is essential for understanding system behavior. In finite-dimensional settings, controllability scores---namely, the Volumetric Controllability Score (VCS) and the Average Energy Controllability Score (AECS)---are defined as the unique solutions of specific optimization problems. In this work, we extend these concepts to infinite-dimensional systems by formulating analogous optimization problems. 
Moreover, we prove that these optimization problems have optimal solutions under weak assumptions, and that both VCS and AECS remain unique in the infinite-dimensional context under appropriate assumptions.
The uniqueness of the controllability scores is essential to use them as a centrality measure, since it not only reflects the importance of each state in the dynamical network but also provides a consistent basis for interpretation and comparison across different researchers.
Finally, we illustrate the behavior of VCS and AECS with a numerical experiment based on the heat equation.
\end{abstract}



\section{Introduction}
\label{sec: introduction}

\subsection{Background}
\label{subsec:background}

Finite-dimensional systems are often regarded as networks with a finite number of nodes. In this context, selecting appropriate control nodes is crucial for maximizing system performance. Existing approaches for such systems can be broadly classified into quantitative methods (see, for example, \cite{ClarkAlomairBushnellPoovendran2017submodularity, pasqualetti2014controllability, RomaoMargellosPapachristodoulou2018, SatoTakeda2020, SummersCortesiLygeros2016})
and qualitative methods (see, for example,\cite{ClarkAlomairBushnellPoovendran2017selection, Olshevsky2014, PequitoKarAguiar2016, satoterasakigraph2024}),
both of which are useful for selecting control nodes. 
In particular, controllability scores---originally introduced as a qualitative centrality measure \cite{satoterasaki2024} and successfully applied in brain network analysis \cite{SatoKawamura2024}---are defined in two forms: the Volumetric Controllability Score (VCS) and the Average Energy Controllability Score (AECS). VCS assesses controllability by measuring the volume of the set of reachable states, while AECS evaluates it based on the average energy required for control. Both scores are derived from optimization problems that quantify the system’s controllability through the spectral properties of the controllability Gramian. In this setting, each component of the state is interpreted as a node, and the corresponding centrality reflects the importance of these nodes. The underlying idea is that the significance of a node can be quantified by determining the weight of a virtual input that maximizes system controllability---a concept implemented by introducing a diagonal input matrix and formulating an associated optimization problem.

Many practical systems, however, are governed by partial differential equations (PDEs) that model spatially distributed phenomena, resulting in an inherently infinite-dimensional structure. For instance, in beam control \cite{EndoSasakiMatsuno2017}, PDEs describe the propagation of vibrations along a continuous structure, while in epidemic modeling \cite{KuniyaJinliangHisashi2016}, they capture the spatial dynamics of disease spread. Because these systems are naturally formulated in infinite-dimensional spaces, traditional centrality measures based on finite-dimensional assumptions are inadequate. This gap highlights the need for novel centrality measures that are specifically designed to assess the influence of system components in infinite-dimensional settings, thereby extending the applicability of controllability-based approaches to a broader class of real-world systems.

\subsection{Contribution}
\label{subsec:contribution}
Based on this background, we extend the concept of controllability scores---namely, the Volumetric Controllability Score (VCS) and the Average Energy Controllability Score (AECS)---to infinite-dimensional systems by formulating them as the unique solutions of specific optimization problems. Although these scores are defined as the optimal solutions of the corresponding problems, their existence and uniqueness are not immediately evident. 
Since both properties are critical for employing controllability scores as reliable centrality measures, we rigorously prove their existence under a set of weak assumptions and establish their uniqueness by incorporating additional appropriate assumptions.
This extension broadens the applicability of controllability-based centrality measures and provides a rigorous framework for assessing node importance even in systems where discrete node interpretations are not straightforward. 
Additionally, we demonstrate the practical applicability of our approach through a numerical experiment on the heat equation. In this experiment, all the required assumptions are satisfied, thereby ensuring the existence and uniqueness of the score and validating our theoretical results. Furthermore, the results showed that VCS tended to evaluate each state to the same degree, while AECS tended to evaluate differences in importance between different states.

\subsection{Outline}
\label{subsec:outline}

The remainder of this paper is organized as follows. In Section 2, we introduce the controllability scores for finite-dimensional systems and discuss key properties that are essential for extending the concept to infinite-dimensional settings and for computing eigenvalues of operators. In Section 3, we extend these controllability scores to infinite-dimensional systems and prove their uniqueness under certain assumptions; we also demonstrate that the scores assume a specific value under a special case. In Section 4, we present a numerical experiment to illustrate the application of the controllability scores. Finally, Section 5 offers concluding remarks.

\section{Preliminaries}
\label{sec:preliminaries}

\subsection{Notation}
\label{subsec:notation}
The set of all positive integers and the set of all real numbers are denoted by $\mathbb{N}$ and $\mathbb{R}$, respectively.
Let $\mathcal{H}$ denote a separable real Hilbert space, where $\|\cdot\|$ and $\langle \cdot,\cdot\rangle$ are its norm and its inner product, respectively.
Let $\{e_i\}_{i=1}^\infty$ denote a complete orthonormal system of $\mathcal{H}$.
Let $P_i$ denote the projection to $\Span\{e_i\}$, defined by $P_i z \coloneq \langle z,e_i\rangle e_i$.
The symbol $L^2([0,T];\mathcal{H})$ denotes the set $L^2([0,T];\mathcal{H}) = \{u:[0,T]\to\mathcal{H} \mid \int_0^T\|u(t)\|^2 \mathrm{d}t <\infty\}$.
For $S\subset \mathcal{H}$, $\Span S$ denotes the linear hull of $S$.
For a linear subspace $V\subset\mathcal{H}$, $V^\bot = \{ w\in \mathcal{H}\mid \langle w,v\rangle = 0 \textrm{ for any $v\in V$}\}$ denotes its orthogonal complement.

For a linear operator $A$, $\dom(A)$ and $\ran(A)$ denotes the domain and range of $A$ respectively.
and $A^*$ denotes the conjugate of $A$.
Moreover, for a linear operator $A$, $\exp(tA)$ denotes the strongly continuous semigroup generated by $A$.

For a matrix $A\in\mathbb{R}^{n\times n}$, $\det A$ and $\tr A$ denote the determinant of $A$ and the trace of $A$, respectively.

\subsection{Controllability score for finite-dimensional systems}
\label{subsec:finite_controllability_score}

In this section, we summarize the idea of the controllability score introduced in \cite{satoterasaki2024}.

For the purposes of this section, we consider the system 
\begin{equation}\label{eq:state_equation_finite_dimension_no_input}
    \dot{x}(t) = Ax(t),
\end{equation}
where $x(t)\in \mathbb{R}^n$ is a state vector whose components represent nodes, and $A\in\mathbb{R}^{n\times n}$ is a matrix representing a weighted network structure.
Our goal is to evaluate the quantitative importance of these nodes. To this end, we augment to system \eqref{eq:state_equation_finite_dimension_no_input} with a virtual weighted input:
\begin{equation}\label{eq:input_network}
    \dot{x}(t) = Ax(t) + Bu(t),
\end{equation} 
where $B \coloneq \diag(\sqrt{p_1},\sqrt{p_2},\dots,\sqrt{p_n})$ is a diagonal matrix encoding the weights of the nodes, and $u(t)\in\mathbb{R}^n$ represents the virtual input.

We develop quantitative measures that capture how easily the system can be controlled.
The first measure is based on the volume of the reachable space, and
the second is based on the energy required to reach the unit sphere.
These measures depend on the weights $p_1,p_2,\dots,p_n$.
Using the measures, we can formulate the optimization problem to maximize the controllability. 
The optimal solutions can be interpreted as measures of node importance, owing to the one-to-one correspondence between the state and input nodes.
However, without any restrictions on the weights, the optimal solution may become unbounded. To prevent this, we impose the constraint $\sum_{i=1}^n p_i = 1$.

To introduce the quantitative measures, we first define the finite-time controllability Gramian for system \eqref{eq:input_network} as 
\begin{equation}
    W(p,T) = \int_0^T \exp(tA)\diag(p_1,p_2,\dots,p_n)\exp(tA^\top)\mathrm{d}t.
\end{equation}
Moreover, if $A$ is stable, we can consider the controllability Gramian of $T\to\infty$ defined as 
\begin{equation}
    W(p) = \int_0^\infty \exp(tA)\diag(p_1,p_2,\dots,p_n)\exp(tA^\top)\mathrm{d}t.
\end{equation}
By definition, the controllability Gramian is positive semidefinite.
In particular, it is known the system is controllable if and only if  the controllability Gramian is positive definite.
For the remainder of this section, we assume $A$ is stable.
\begin{lemma}[\cite{zabczyk2020mathematical}, Proposition 1.1.]\label{lem:minimum_energy_to_fixed_point_finite_dimension}
    If system \eqref{eq:input_network} is controllable, 
    \begin{equation}
        \min \left\{\left.\int_0^T \|u(t)\|^2 \mathrm{d}t \right\vert x(0) = 0, x(T) = x_f, u\in L^2([0,T];\mathbb{R}^n) \right\}
        = x_f^\top W(p,T)^{-1}x_f,
    \end{equation}
    where $x_f$ is any point in $\mathbb{R}^n$.
\end{lemma}

The next Lemma is known as the direct consequence of Lemma \ref{lem:minimum_energy_to_fixed_point_finite_dimension}.
\begin{lemma} \label{Lem_reachable_ellipsoid}
    If system \eqref{eq:input_network} is controllable, 
    \begin{equation}
        \mathcal{E}(T) \coloneq
        \left\{
            x(T)\in\mathbb{R}^n \middle|
            \begin{aligned}
                u\in L^2[0,T], \int_0^T \|u(t)\|^2 \mathrm{d}t \leq 1\\
                \dot{x}(t) = Ax(t)+Bu(t), x(0)=0
            \end{aligned}
        \right\}
        = \{ x\in\mathbb{R}^n \mid x^\top W(p,T)^{-1}x \leq 1\}.
    \end{equation}
    Moreover, the volume of $\mathcal{E}(T)$ is $V_n (\det W(p,T))^{1/2}$, where $V_n$ is the volume of the unit ball in $n$ dimensional space, as shown in \cite[Corollary 12.15]{vishnoi2021algorithms}.
\end{lemma}

Lemma \ref{Lem_reachable_ellipsoid} establishes that if system \eqref{eq:input_network} is controllable, then the reachable set $\mathcal{E}(T)$---the set of all states reachable at time $T$ with an input of bounded energy---is an ellipsoid. 
A key factor in determining the volume of $\mathcal{E}(T)$ is $\det W(p,T)$; 
a larger value of $\det W(p,T)$ indicates that $\mathcal{E}(T)$ is larger and, consequently, that the system is easier to control.
Furthermore, by considering the limit as $T\rightarrow \infty$, we can also regard $\det W(p)$ as a candidate for a centrality measure.

Based on this idea, we construct the following minimization problem, whose optimal solution is defined as the Volumetric Controllability Score (VCS).

\begin{framed}
\begin{align}
    \begin{aligned}
        &&& \text{minimize} && f(p) := -\log\det W(p) \\
        &&& \text{subject to} && p \in X\cap \Delta.
    \end{aligned}
\end{align}
\end{framed}
\noindent
Here,
\begin{equation}\label{eq:def_X}
    X \coloneq \left\{ p\in\mathbb{R}^n \mid W(p) \succ O\right\},
\end{equation}
\begin{equation}
    \Delta \coloneq \left\{p\in\mathbb{R}^n \left\vert \sum_{i=1}^n p_i = 1, p_i\geq 0\right.\right\}.
\end{equation}
Note that $p\in X$ is equivalent to system (\ref{eq:input_network}) being controllable.
This is necessary to define $f(p)$.

If system (\ref{eq:input_network}) is controllable, the objective function $f(p)$ can be written by using eigenvalues of $W(p)$.
\begin{lemma}
    Let $\lambda_1,\lambda_2,\dots,\lambda_n$ be the eigenvalues of $W(p)$.
    If $p\in X$, $W(p)$ is positive definite, and $\lambda_1,\lambda_2,\dots,\lambda_n >0$.
    Moreover, $\det W(p) = \lambda_1 \lambda_2\cdots\lambda_n$.
    Therefore, 
    \begin{equation}
        \begin{split}
            f(p) = -\log \prod_{k=1}^n \lambda_k.
        \end{split}
    \end{equation}
\end{lemma}

Next, we construct the second score, which is based on the average energy required for control. 
Lemma \ref{lem:minimum_energy_to_fixed_point_finite_dimension} suggests that by considering the average energy needed to reach the unit sphere, we obtain a measure that reflects the system's energy controllability. 
The following lemma presents a key fact used to derive the measure, as explained in \cite{Olshevsky2014}.

\begin{lemma}\label{lem:average_energy}
    Suppose that $x_f$ is uniformly distributed on the unit sphere. Then,
    \begin{equation}
        \mathrm{E}_{x_f}[x_f^\top W(p,T)^{-1}x_f] = \frac{1}{n}\tr (W(p,T)^{-1})
    \end{equation}
    for any $p\in X$.
\end{lemma}

Lemmas \ref{lem:minimum_energy_to_fixed_point_finite_dimension} and \ref{lem:average_energy} show that when the target state
$x_f$ is uniformly distributed on the unit sphere, the average minimal energy required to steer the system from the origin to
$x_f$ can be expressed in a simple form involving the controllability Gramian.
A key factor for calculating the average minimal energy is $\tr (W(p,T)^{-1})$; a lower value indicates that the system is easier to control over the unit sphere. Moreover, by considering the limit as $T\to\infty$, we can also regard $\tr (W(p)^{-1})$ as a candidate for a centrality measure.

Based on this, we can formulate a minimization problem.
The optimal solution to this problem is called AECS (Average Energy Controllability Score).
\begin{framed}
    \begin{align}
        \begin{aligned}
            &&& \text{minimize} && g(p) := \tr(W(p)^{-1}) \\
            &&& \text{subject to} && p \in X\cap \Delta.
        \end{aligned}
    \end{align}
\end{framed}

The function $g(p)$ is written with eigenvalues of $W(p)$.
\begin{lemma}
    Let $\lambda_1,\lambda_2,\dots,\lambda_n$ be the eigenvalues of $W(p)$.
    If $p\in X$, $W(p)$ is positive definite, and $\lambda_1,\lambda_2,\dots,\lambda_n >0$.
    Moreover, the eigenvalues of $W(p)^{-1}$ are $\frac{1}{\lambda_1}, \frac{1}{\lambda_2}, \dots, \frac{1}{\lambda_n}$.
    Therefore,
    \begin{equation}
        \begin{split}
            g(p) = \sum_{k=1}^n \frac{1}{\lambda_k}.
        \end{split}
    \end{equation}
\end{lemma}

\section{Controllability scores for infinite-dimensional systems}\label{section:controllability_score_infinite_system}

As explained in Section \ref{subsec:finite_controllability_score},
controllability scores were defined only for finite dimensional systems in \cite{satoterasaki2024}.
In this section, we extend the definitions of controllability scores to infinite-dimensional systems.

\subsection{Optimization Problems for VCS and AECS}\label{sec:definition_of_controllability_score}

From this section, we consider
\begin{equation}\label{eq:system_equation_infinite_no_input}
    \begin{split}
        \dot{x}(t) = Ax(t) \quad \textrm{for $x(t)\in\dom(A)$},
    \end{split}
\end{equation} 
where $A$ is an infinitesimal generator of a strongly continuous semigroup on $\mathcal{H}$, which
 $\mathcal{H}$ is a separable real Hilbert space, that serves as a state space.
We assume that $A$ is exponentially stable. In other words, there exist constants $M,\omega >0$ 
such that $\|\exp(tA)\|\leq M\exp(-\omega t)$ holds for all $t\geq 0$.

As mentioned in Section \ref{subsec:finite_controllability_score}, for finite-dimensional systems, a one-to-one virtual input was added.
This was represented by a diagonal matrix $B$, which satisfies $Be_i = \sqrt{p_i}e_i$ for the standard basis $e_i\in\mathbb{R}^n$, as shown in (\ref{eq:input_network}).
In this case, $e_i$ corresponds to a state node and the importance of $e_i$ is considered.
We adopt this idea.

For infinite systems, we can consider a complete orthonormal system $\{e_i\}_{i=1}^\infty$ as state nodes, with the aim of evaluating centrality.
For example, if we want to control the temperature of a metal rod, adding heat at a specific point is one way of controlling it.
This can be expressed by standardized simple functions, and if the points we control are disjoint, the corresponding simple functions form an orthonormal system in $L^2$.

Thus, we associate state nodes with a complete orthonormal system $\{e_i\}_{i=1}^\infty$.
Moreover, we add a virtual input $\sqrt{p_i}e_i$ to the system, where $\{p_i\}_{i=1}^\infty$ satisfies $\sum_{i=1}^\infty p_i = 1$ and $p_i\geq 0$,
and define a linear operator $B$ as $Be_i = \sqrt{p_i}e_i$ for $i\in\mathbb{N}$.
Each $p_i$ represents the weight of the input.
Then, we augment system \eqref{eq:system_equation_infinite_no_input} with a
 virtual input $u(t)$:
\begin{equation}\label{eq:state_equation_infinite_with_input_diagonal_B}
    \dot{x}(t) = Ax(t) + Bu(t) \quad\textrm{for $x(t)\in\dom(A)$}.
\end{equation}

We introduce the reachability operator and controllability Gramian to characterize the controllability of system \eqref{eq:state_equation_infinite_with_input_diagonal_B} in the presence of the virtual input.

\begin{definition}
    For $T\in[0,\infty)$, we define the reachability operator $L(p,T):L^2([0,T];\mathcal{H})\to\mathcal{H}$ be an operator defined by 
    $
        L(p,T)u = \int_0^T \exp((T-t)A)Bu(t) \mathrm{d}t.
    $
    In addition, we define the controllability Gramian of system \eqref{eq:state_equation_infinite_with_input_diagonal_B} as 
    \begin{equation}
    W(p,T) = \int_0^T \exp(tA)BB^* \exp(tA)^* \mathrm{d}t.
\end{equation}
    Moreover, since $A$ is exponentially stable and $B$ is a bounded operator, $W(p,T)$ has a limit with $T\to\infty$,
    \begin{equation}
    W(p) \coloneq W(p,\infty) = \int_0^\infty \exp(tA)BB^* \exp(tA)^* \mathrm{d}t.
\end{equation}
\end{definition}

We now establish a key property of $W(p,T)$.

\begin{lemma}\label{lem:compactness_of_controllability_Gramian}
    The controllability Gramian $W(p,T)$ is a compact operator for any $T\in [0,\infty]$ and $\{p_i\}_{i=1}^\infty$ with $\sum_{i=1}^\infty p_i = 1$, $p_i\geq 0$.
    Furthermore, there exist a sequence of nonnegative real numbers $\{\lambda_k\}_{k=1}^\infty$ and a complete orthonormal system $\{z_k\}_{k=1}^\infty$ such that 
    \begin{equation}\label{eq:specral_decomposition_W}
        W(p,T)z = \sum_{k=1}^\infty \lambda_k \langle z,z_k\rangle z_k.
    \end{equation}
    Here, $T\in[0,\infty]$ means $T\in[0,\infty)$ or $T=\infty$.
\end{lemma}
\begin{proof}
    First, we consider the case $T\in [0,\infty)$.
    The conjugate of $L(p,T)$ is given by ${L(p,T)}^* = B^*\exp((T-t)A)^*$.
    Since $\sum_{i=1}^\infty \|Be_i\|^2 = \sum_{i=1}^\infty p_i = 1$, $B$ is a Hilbert-Schmidt operator, and thus $B$ is a compact operator \cite[Theorem VI.22]{reed1980methods}.
    Furthermore, the product of a compact operator and a bounded operator is compact,
    that implies ${L(p,T)}^*$ and $L(p,T) {L(p,T)}^*$ are also compact.
    The finite-time controllability Gramian has the form of 
    \begin{equation}\label{eq:controllability_Gramian_conjugate_form}
        W(p,T) = L(p,T) {L(p,T)}^*.
    \end{equation}
    This implies that $W(p,T)$ is a compact operator.
    Moreover, since $W(p,T)$ converges to $W(p)$ with respect to 
    the operator norm as $T\to\infty$, $W(p)$ is also compact. By \eqref{eq:controllability_Gramian_conjugate_form},
    the controllability Gramian is self-adjoint and positive semidefinite.
    Therefore, the spectral decomposition of $W(p,T)$ has the form of \eqref{eq:specral_decomposition_W}
\end{proof}

From Lemma \ref{lem:compactness_of_controllability_Gramian}, we can order the eigenvalues of $W(p,T)$ in descending order.
We write $k$-th eigenvalue of $W(p,T)$ as $\mu_k(W(p,T))$.

The following lemma is introduced to present some properties related to controllability.
\begin{lemma}\label{lem:orthogonal_projection_P_W_n_dagger}
    Suppose $W(p,T)z = \sum_{k=1}^\infty \lambda_k \langle z,z_k\rangle z_k$, and suppose $\lambda_1,\lambda_2,\dots,\lambda_n>0$.
    Then, define $W_n^\dagger(p,T)$ as 
\begin{equation}\label{eq:definition_of_Wn_dagger}
        W_n^\dagger(p,T) z = \sum_{k=1}^n \frac{1}{\lambda_k}\langle z,z_k\rangle z_k
    \end{equation}
    and $P(p,T)$ as
    \begin{equation}\label{eq:definition_of_projection_P}
        P(p,T) = L(p,T)^* W_n^\dagger(p,T) L(p,T).
    \end{equation}
    Then, $P(p,T)$ is an orthogonal projection on $L^2([0,T];\mathcal{H})$ for any $T\in[0,\infty)$.
\end{lemma}
\begin{proof}
    Since $L(p,T)$ and $W_n^\dagger(p,T)$ are bounded operators, $P(p,T)$ is bounded.
    In addition, it is straightforward to confirm $P(p,T)^* = P(p,T)$.
    Moreover,
    \begin{equation}
        \begin{split}
            P(p,T)^2
            &= L(p,T)^* W_n^\dagger(p,T)L(p,T)L(p,T)^* W_n^\dagger(p,T)L(p,T)\\
            &= L(p,T)^* W_n^\dagger(p,T)W(p,T) W_n^\dagger(p,T)L(p,T)= P(p,T),
        \end{split}
    \end{equation}
    where the last equality follows from $W_n^\dagger(p,T)W(p,T)W_n^\dagger(p,T) = W_n^\dagger(p,T)$.
\end{proof}

Using Lemma \ref{lem:orthogonal_projection_P_W_n_dagger}, we first introduce the measure that corresponds to VCS in finite-dimensional systems.
To this end, we consider the reachable space of system \eqref{eq:state_equation_infinite_with_input_diagonal_B} defined as
\begin{equation}
    \mathcal{E}(T) \coloneq
    \left\{
        L(p,T)u \in\mathcal{H} \middle|
        \begin{aligned}
            u\in L^2([0,T];\mathcal{H}), \int_0^T \|u(t)\|^2 \mathrm{d}t \leq 1
        \end{aligned}
    \right\}
\end{equation}
and show the following proposition.

\begin{proposition}\label{prop:VCS_intersection}
    Let $W(p,T)z = \sum_{k=1}^\infty \lambda_k\langle z,z_k\rangle z_k$
    If $\lambda_1,\lambda_2,\dots,\lambda_n>0$, the intersection of the reachable space and the linear span of $\{z_1,z_2,\dots,z_n\}$ has the following form.
    \begin{equation}\label{eq:VCS_intersection}
        \begin{split}
            &\mathcal{E}(T) \cap \Span\{z_1,z_2,\dots,z_n\}\\
            &= \left\{ a_1 z_1 + a_2 z_2 + \dots + a_n z_n \left\vert \frac{{a_1}^2}{\lambda_1} + \frac{{a_2}^2}{\lambda_2} + \dots + \frac{{a_n}^2}{\lambda_n} \leq 1\right.\right\}.
        \end{split}
    \end{equation}
\end{proposition}
\begin{proof}
    First, we show 
    \begin{equation}\label{eq:E_T_subset_less_than_1}
        \begin{split}
            &\mathcal{E}(T) \cap \Span\{z_1,z_2,\dots,z_n\}\\
            &\subset \left\{ a_1 z_1 + a_2 z_2 + \dots + a_n z_n \left\vert \frac{{a_1}^2}{\lambda_1} + \frac{{a_2}^2}{\lambda_2} + \dots + \frac{{a_n}^2}{\lambda_n} \leq 1\right.\right\}.
        \end{split}
    \end{equation}
    Let $x\in \mathcal{E}(T) \cap \Span\{z_1,z_2,\dots,z_n\}$.
    Since $x\in\Span\{z_1,z_2,\dots,z_n\}$, there exists $\{a_k\}_{k=1}^n$ such that 
    $x = \sum_{k=1}^n a_k z_k$.
    Thus,
    \begin{equation}\label{eq:W_dagger_equal_sum_of_inner_product}
        \langle W_n^\dagger(p,T) x, x\rangle 
        = \left\langle \sum_{k=1}^n \frac{a_k}{\lambda_k}z_k,\sum_{k=1}^n a_k z_k\right\rangle
        = \sum_{k=1}^n \frac{{a_k}^2}{\lambda_k}.
    \end{equation}
    Therefore, it is sufficient to show $\langle W_n^\dagger (p,T)x,x\rangle\leq 1$.
    Since $x\in \mathcal{E}(T)$, there exists $u\in L^2([0,T];\mathcal{H})$ such that $\int_0^T \|u(t)\|^2\mathrm{d}t \leq 1$ and $x = L(p,T) u$.
    Then,
    \begin{equation}\label{eq:W_n_dagger_less_than_1}
        \begin{split}
            \langle W_n^\dagger(p,T)x,x\rangle
            &= \langle L(p,T)^*W_n^\dagger(p,T)L(p,T)u, u\rangle\\
            &= \langle P(p,T)u,u\rangle
            = \|P(p,T)u\|^2\leq \|u\|^2
            \leq 1,
        \end{split}
    \end{equation}
    where $P(p,T)$ is defined as \eqref{eq:definition_of_projection_P} in Lemma \ref{lem:orthogonal_projection_P_W_n_dagger} and we used the fact that $P(p,T)$ is an orthogonal projection.
    Thus, equation \eqref{eq:E_T_subset_less_than_1} holds.
    
    Next, we show 
    \begin{equation}\label{eq:E_T_supset_less_than_1}
        \begin{split}
            &\left\{ a_1 z_1 + a_2 z_2 + \dots + a_n z_n \left\vert \frac{{a_1}^2}{\lambda_1} + \frac{{a_2}^2}{\lambda_2} + \dots + \frac{{a_n}^2}{\lambda_n} \leq 1\right.\right\}\\
            &\subset \mathcal{E}(T)\cap \Span\{z_1,z_2,\dots,z_n\}.
        \end{split}
    \end{equation}
    Let $v\in \{ a_1 z_1 + a_2 z_2 + \dots + a_n z_n \mid \frac{{a_1}^2}{\lambda_1} + \frac{{a_2}^2}{\lambda_2} + \dots + \frac{{a_n}^2}{\lambda_n} \leq 1\}$.
    Since $\lambda_i z_i =W(p,T)z_i$ and $\lambda_i >0$, $z_i\in \ran W(p,T)$.
    Moreover, since $W(p,T) = L(p,T)L(p,T)^*$, $z_i\in \ran L(p,T)$.
    This implies that there exists $u\in L^2([0,T];\mathcal{H})$ such that $v = L(p,T)u$.
    Then, as in \eqref{eq:W_n_dagger_less_than_1}, it holds
    \begin{equation}\label{eq:W_n_dagger_less_than_1_same}
        1\geq \langle W_n^\dagger(p,T)v,v\rangle = \|P(p,T)u\|^2.
    \end{equation}
    Here, it holds
    \begin{equation}
        L(p,T)P(p,T)u = L(p,T)L(p,T)^*W_n^\dagger(p,T)L(p,T)u = W(p,T)W_n^\dagger(p,T)v.
    \end{equation}
    Since $v\in \Span\{z_1,z_2,\dots,z_n\}$, $W(p,T)W_n^\dagger(p,T)v = v$.
    This implies that $v = L(p,T)P(p,T)u$. 
    By equation \eqref{eq:W_n_dagger_less_than_1_same}, $v \in \mathcal{E}(T)$.
    Therefore, equation \eqref{eq:E_T_supset_less_than_1} holds.
\end{proof}

The volume of $\{(a_1,a_2,\dots,a_n)\in\mathbb{R}^n \mid \frac{{a_1}^2}{\lambda_1} + \frac{{a_2}^2}{\lambda_2} + \dots + \frac{{a_n}^2}{\lambda_n}\leq 1\}$ in Proposition \ref{prop:VCS_intersection} is $V_n\sqrt{\lambda_1\lambda_2\dots\lambda_n}$, where $V_n$ is the volume of $n$ dimensional unit ball.
Therefore, we can consider $n$ dimensional volume of $\mathcal{E}(T)\cap\Span\{z_1,z_2,\dots,z_n\}$ as $V_n\sqrt{\lambda_1\lambda_2\dots\lambda_n}$.
Furthermore, this volume has its maximum if $\lambda_k = \mu_k(W(p,T))$ ($k=1,2,\dots,n$), where $\mu_k(W(p,T))$ is the $k$-th eigenvalue of $W(p,T)$.
Thus, we can consider $\prod_{k=1}^n \mu_k(W(p,T))$ as a measure of centrality.
In addition, since $A$ is exponentially stable, we can consider $\mathcal{E}(T)$ with $T\to\infty$.
Therefore, we use $\prod_{k=1}^n \mu_k(W(p))$ as a measure of controllability.

Thus, we consider the following minimization problem, whose optimal solution is called the Volumetric Controllability Score (VCS).
\begin{framed}
    \begin{align}\label{problem:VCS_infinite}
        \begin{aligned}
            &&& \text{minimize} && f(p) \coloneq -\log\prod_{k=1}^n \mu_k(W(p))\\
            &&& \text{subject to} && \{p_i\}_{i=1}^\infty \in \Delta\cap X.
        \end{aligned}
    \end{align}
\end{framed}
\noindent
Here,
\begin{align*}
 \Delta = \left\{\{p_i\}_{i=1}^\infty \left\vert \sum_{i\in I}p_i = 1, p_i\geq 0\right.\right\},\quad
X = \left\{\{p_i\}_{i=1}^\infty \mid \mu_n(W(p))>0\right\},
\end{align*}
and $n$ is a positive integer which can be chosen arbitrarily.
Note that in Problem \eqref{problem:VCS_infinite}, we cannot consider the case where $n\to\infty$. 
In fact, since $W(p)$ is a compact operator from Lemma \ref{lem:compactness_of_controllability_Gramian} and 
$\mu_n(W(p))$ is an eigenvalue of $W(p)$, we have $\mu_n(W(p))\rightarrow 0$ as $n\to\infty$.
Thus, $f(p)$ diverges as $n\to\infty$.
However, for any finite $n$, $f(p)$ is well-defined.
Therefore,  $n$ can be chosen arbitrarily from $\mathbb{N}$.

We next introduce another measure that corresponds to AECS.
 To this end, 
we use Lemma \ref{lem:orthogonal_projection_P_W_n_dagger} to prove the following proposition.

\begin{proposition}\label{prop:AECS_reachable_energy_span}
    Let $W(p,T)z = \sum_{k=1}^\infty \lambda_k \langle z,z_k\rangle z_k$.
    Then, the minimum energy required to reach $x_f\coloneq a_1z_1 + a_2z_2 + \dots + a_n z_n$ is the optimal value of the following problem.
    \begin{align}\label{eq:minimal_energy_problem_infinite_span_z}
        \begin{aligned}
            &&& \text{minimize} && \int_0^T \|u(t)\|^2 \mathrm{d}t\\
            &&& \text{subject to} && L(p,T)u = x_f
        \end{aligned}
    \end{align}
    If $\lambda_1,\lambda_2,\dots,\lambda_n>0$, the optimal value of Problem \eqref{eq:minimal_energy_problem_infinite_span_z} is 
    $\frac{{a_1}^2}{\lambda_1} + \frac{{a_2}^2}{\lambda_2} + \dots + \frac{{a_n}^2}{\lambda_n}$.
\end{proposition}
\begin{proof}
    Let $u_{\text{opt}} = L(p,T)^* W_n^\dagger(p,T)x_f$, where $W_n^{\dagger}(p,T)$ is defined as \eqref{eq:definition_of_Wn_dagger} in Lemma \ref{lem:orthogonal_projection_P_W_n_dagger}. 
    Then, $L(p,T)u_{\text{opt}} = L(p,T)L(p,T)^* W_n^\dagger(p,T)x_f = W(p,T)W_n^\dagger(P,T)x_f = x_f$
    since $x_f \in \Span\{z_1,z_2,\dots,z_n\}$.
    Therefore, $u_{\text{opt}}$ satisfies the constraint condition.
    Moreover, if $u$ satisfies the constraint condition,
    \begin{equation}
        \begin{split}
            P(p,T)u
            = L(p,T)^*W_n^\dagger(p,T)L(p,T)u
            = L(p,T)^*W_n^\dagger(p,T)x_f
            = u_{\text{opt}},
        \end{split}
    \end{equation}
    where $P(p,T)$ is defined as \eqref{eq:definition_of_projection_P} in Lemma \ref{lem:orthogonal_projection_P_W_n_dagger}.
    Since $P(p,T)$ is an orthogonal projection, $\|u_{\text{opt}}\| = \|P(p,T)u\| \leq \|u\|$.
    This implies $u_{\text{opt}}$ is an optimal solution of Problem \eqref{eq:minimal_energy_problem_infinite_span_z}.
    Furthermore, its optimal value is 
    \begin{equation}
        \begin{split}
            \langle u_{\text{opt}},u_{\text{opt}}\rangle
            &= \langle L(p,T)^*W_n^\dagger(p,T) x_f, L(p,T)^*W_n^\dagger(p,T) x_f\rangle\\
            &= \langle W_n^\dagger(p,T)x_f,x_f\rangle
            = \frac{{a_1}^2}{\lambda_1} + \frac{{a_2}^2}{\lambda_2} + \dots + \frac{{a_n}^2}{\lambda_n}.
        \end{split}
    \end{equation}
\end{proof}

If $(a_1,a_2,\dots,a_n)$ follows the uniform distribution on the $n$ dimensional unit sphere, the expected value of $\frac{{a_1}^2}{\lambda_1} + \frac{{a_2}^2}{\lambda_2} + \dots + \frac{{a_n}^2}{\lambda_n}$
is $\frac{1}{n} \left(\frac{1}{\lambda_1} + \frac{1}{\lambda_2} + \dots + \frac{1}{\lambda_n}\right)$.
Moreover, this value takes the minimum if $\lambda_k = \mu_k(W(p,T))$ ($k=1,2,\dots,n$), where $\mu_k(W(p,T))$ is the $k$-th eigenvalue of $W(p,T)$.
Therefore, we can consider $\frac{1}{\mu_1(W(p,T))} + \frac{1}{\mu_2(W(p,T))} + \dots + \frac{1}{\mu_2(W(p,T))}$ as a centrality measure.
In addition, since $A$ is exponentially stable, we can consider $T\to\infty$ and use $\frac{1}{\mu_1(W(p))} + \frac{1}{\mu_2(W(p))} + \dots + \frac{1}{\mu_n(W(p))}$ as a measure.

Proposition \ref{prop:AECS_reachable_energy_span} leads us to the following minimization problem, whose optimal solution is called the Average Energy Controllability Score (AECS).

\begin{framed}
    \begin{align}\label{problem:AECS_infinite}
        \begin{aligned}
            &&& \text{minimize} && g(p) \coloneq \sum_{k=1}^n \frac{1}{\mu_k(W(p))} \\
            &&& \text{subject to} && \{p_i\}_{i=1}^\infty \in \Delta\cap X.
        \end{aligned}
    \end{align}
\end{framed}
\noindent
Here, $n$ is a finite positive integer that can be chosen arbitrarily.
Similarly to Problem \eqref{problem:VCS_infinite}, in Problem \eqref{problem:AECS_infinite}, we cannot consider the case where $n\to\infty$ because $g(p)$ diverges as $n\to\infty$. However, as long as $g(p)$ is well-defined,  any finite $n\in\mathbb{N}$ is acceptable.

\subsection{Existence of VCS and AECS}\label{sec:existence_of_VCS_and_AECS}

Problems \eqref{problem:VCS_infinite} and \eqref{problem:AECS_infinite} have an optimal solution under the following assumptions.

\begin{assumption}\label{ass:p_is_less_than_a}
    Let $\{a_i\}_{i=1}^\infty$ satisfy $\sum_{i=1}^\infty a_i <\infty$ and $a_i\geq 0$.
    We assume $0\leq p_i \leq a_i$ for all $i\in\mathbb{N}$.
\end{assumption}

\begin{assumption}\label{ass:feasible_solution}
    There exist feasible solutions to Problems \eqref{problem:VCS_infinite} and \eqref{problem:AECS_infinite}. 
    That is, there exists $p\in \Delta\cap X$ that satisfies Assumption \ref{ass:p_is_less_than_a}.
\end{assumption}

To illustrate Assumptions \ref{ass:p_is_less_than_a} and \ref{ass:feasible_solution}, we now present an example based on the heat equation.

\begin{example}
Consider the heat equation with Dirichlet boundary conditions in $L^2(0, 1)$,
    $\dot{u}(t) = Au(t)$,
where $Au(t) = \frac{\mathrm{d}^2 u}{\mathrm{d}x^2}$.
When controlling the system, the location at which heat is applied significantly influences its performance. To capture this spatial effect, we introduce the following virtual input: For $k\in\mathbb{N}$,
 \begin{equation}
    \begin{split}
     &e_{2k-1}(x) = 
     \begin{cases}
         \sqrt{2\sqrt{2}}k & \text{if $x\in (\frac{1}{2} + \sum_{m=1}^{k-1} \frac{\sqrt{2}}{4m^2}, \frac{1}{2} + \sum_{m=1}^k \frac{\sqrt{2}}{4m^2})$}\\
        0 & \text{otherwise}
     \end{cases},
     \\
     &e_{2k}(x) = 
     \begin{cases}
         \sqrt{2\sqrt{2}}k & \text{if $x\in (\frac{1}{2} - \sum_{m=1}^{k} \frac{\sqrt{2}}{4m^2}, \frac{1}{2} - \sum_{m=1}^{k-1} \frac{\sqrt{2}}{4m^2})$}\\
        0 & \text{otherwise}
     \end{cases}.
     \end{split}
 \end{equation}
Since $\{e_k\}_{k=1}^\infty$ forms an orthonormal system in $L^2(0,1)$,
we can extend it to a complete orthonormal system by choosing functions $\{f_l\}_{l=1}^\infty$ such that $\{e_k\}_{k=1}^\infty \cup \{f_l\}_{l=1}^\infty$ is complete in $L^2(0, 1)$.
Define an operator $B$ representing the weights by $Be_i = \sqrt{p_k}e_k$, and $Bf_l = \sqrt{q_l}f_l$.
With this virtual input, system $\dot{u}(t) = Au(t)$ becomes
    $\dot{u}(t) = Au(t) + Bv$,
where $v$ represents the input.

To ensure that Assumption \ref{ass:p_is_less_than_a} holds, we define
        $a_{2k} = \frac{1}{k^2}$,
        $a_{2k-1} = \frac{1}{k^2}$
for $k\in\mathbb{N}$.
In addition, we also define $b_l = 0$.
Then, $\sum_{k=1}^\infty a_k + \sum_{l=1}^\infty b_l < \infty$.
Therefore, if we assume $p_k\leq a_k \coloneq \frac{1}{(k+1)^2}$, then, $q_l \leq b_l$ and Assumption \ref{ass:p_is_less_than_a} hold.
Since $b_l = 0$, $q_l = 0$ and this implies that $f_l$ does not affect the system, and we can purely evaluate the centrality of $e_k$.
Furthermore, the sequence $\{a_k\}$ is decreasing, reflecting the physical situation in which more heat can be applied near the center of the interval and less heat near the boundaries.

To ensure that Assumption \ref{ass:feasible_solution} holds,
we consider the case where $p=(p_1,p_2,p_3,\cdots) = (1,0,0,\cdots)$.
Then, the controllability Gramian $W(p)$ is given by 
\begin{equation}
    W(p) = \int_0^\infty \exp(tA)P_1P_1^*\exp(tA)^*\mathrm{d}t,
\end{equation}
where $P_1z = \langle z,e_1\rangle e_1$.
According to \cite[Theorem 6.4.5]{CurtainZwart2020}, the following are equivalent: (i) $W(p)$ is positive definite, (ii) $P_1^*\exp(tA)^*z = 0 \Rightarrow z = 0$ for all $t\geq 0$.
In this example, $\exp(tA)^*z = \exp(tA)z = \sum_{m=1}^\infty e^{-m^2\pi^2 t}\langle z,c_m\rangle c_m$ and $P_1^*\exp(tA)^*z = P_1\exp(tA)z = \sum_{m=1}^\infty e^{-m^2\pi^2t} \langle z,c_m\rangle \langle c_m,e_1\rangle e_1$, where $c_m(x) = \sqrt{2}\sin(m \pi x)$.
Therefore, $P_1^*\exp(tA)^* z = 0$ implies that $\sum_{m=1}^\infty e^{-m^2\pi^2 t} \langle z,c_m\rangle \langle c_m,e_1\rangle  = 0$ for all $t\geq 0$.
Since the functions $e^{-m^2\pi^2 t}$ are linearly independent, $\langle z,c_m\rangle \langle c_m,e_1\rangle = 0$ for all $m\in \mathbb{N}$.
Moreover, a simple calculation shows that $\langle c_m,e_1\rangle \neq 0$ for all $m\in\mathbb{N}$. Hence, we obtain $\langle z,c_m \rangle = 0$ for all $m\in\mathbb{N}$.
This implies $z=0$ because $\{c_m\}_{m=1}^\infty$ is a complete orthonormal system in $L^2(0,1)$.
Thus, $W(p)$ is positive semidefinite.
Furthermore, Lemma \ref{lem:compactness_of_controllability_Gramian} implies that $W(p)$ has a infinite number of positive eigenvaluess, which in turn verifies \ref{ass:feasible_solution}.
\end{example}

The following theorem shows that VCS and AECS exist under Assumptions \ref{ass:p_is_less_than_a} and \ref{ass:feasible_solution}.

\begin{theorem}\label{thm:existence_of_optimal_solution}
    If Assumptions \ref{ass:p_is_less_than_a} and \ref{ass:feasible_solution} hold, then an optimal solution exists for Problems \eqref{problem:VCS_infinite} and \eqref{problem:AECS_infinite}.
\end{theorem}
\begin{proof}
    See Appendix \ref{sec:proof_existence_optimal_solution}.
\end{proof}
\begin{remark}
    For Problems \eqref{problem:VCS_infinite} and \eqref{problem:AECS_infinite}, the continuity of $f(p)$ and $g(p)$ is not apparent unlike the finite dimensional case.
    This is because these functions are computed by selecting only a finite number of eigenvalues from among the infinitely many available, unlike the finite dimensional case.
    Therefore, we prove the continuity in Appendix \ref{sec:proof_existence_optimal_solution}.
\end{remark}

\subsection{Uniqueness of VCS and AECS}\label{sec:uniqueness_of_VCS_and_AECS}

Although we have established the existence of VCS and AECS in Theorem \ref{thm:existence_of_optimal_solution}, if the sequence $\{p_i\}_{i=1}^\infty$ contains infinitely many nonzero elements, proving the uniqueness of the optimal solution becomes challenging.
This poses a significant obstacle to using VCS and AECS as centrality measures. 
Therefore, we restrict the number of virtual inputs in system \eqref{eq:state_equation_infinite_with_input_diagonal_B} to be finite.
This assumption allows us to establish the uniqueness of the optimal solution.
Mathematically, this property can be expressed as follows.

\begin{assumption}\label{ass:input_node}
    Let $I\subset \mathbb{N}$ be a finite set that satisfies $p_i = 0$ for $i\notin I$.
\end{assumption}

Under Assumption \ref{ass:input_node}, $B$ satisfies $Be_i = \sqrt{p_i}e_i$ for $i\in I$, and 
$Be_i = 0$ for $i\notin I$.
Moreover, Assumption \ref{ass:p_is_less_than_a} is satisfied if $a_i = 1$ for $i\in I$ and $a_i = 0$ for $i \notin I$.
Therefore, Assumption \ref{ass:p_is_less_than_a} holds for some $\{a_i\}_{i=1}^\infty$ under Assumption \ref{ass:input_node}.

To prove the uniqueness of VCS and AECS, we introduce the following definition along with an assumption.

\begin{definition} \label{def_Wi}
    Define $W_i$ as 
    \begin{equation}
        W_i = \int_0^\infty \exp(tA)P_iP_i^* \exp(tA)^* \mathrm{d}t,
    \end{equation}
    where $P_i$ is the projection to $\Span \{e_i\}$, defined by $P_i z = \langle z,e_i\rangle e_i$.
\end{definition}

\begin{assumption}\label{ass:commute}
    We assume $\{W_i\}_{i=1}^\infty$ is commutative, which means that $W_i W_j = W_j W_i$ for all $i,j\in\mathbb{N}$.
\end{assumption}

Note that Definition \ref{def_Wi} can be regarded as the special case of the controllability Gramian when $B = P_i$.

If $A$ is self-adjoint and $\{e_i\}_{i=1}^\infty$ are eigenvectors of
$A$, Assumption \ref{ass:commute} holds.
In fact, if $Ae_j = \lambda_j e_j$, it holds
\begin{equation}
    \exp(tA)e_j = \exp(\lambda_j t)e_j.
\end{equation}
Furthermore, it also holds
\begin{equation}
    \begin{split}
        &\exp(tA)P_i P_i^*\exp(tA)^* e_j
        = \exp(tA)P_i \exp(\lambda_j t)e_j\\
        &= \exp(\lambda_j t)\exp(tA)\delta_{ij}e_j
        = \exp(2\lambda_j t)\delta_{ij}e_j.
    \end{split}
\end{equation}
Therefore, $W_i$ satisfies
\begin{equation}\label{eq:W_i_e_j_to_prove_commute}
    \begin{split}
        W_ie_j = \int_0^\infty \exp(2\lambda_j t)\delta_{ij}e_j \mathrm{d}t
        &= \frac{\delta_{ij}}{-2\lambda_j}e_j.
    \end{split}
\end{equation}
This integral converges because $\lambda_j < 0$ due to $\|\exp(tA)\| \leq M\exp(-\omega t)$.
The equation \eqref{eq:W_i_e_j_to_prove_commute} implies 
\begin{equation}\label{eq:W_iW_je_k_to_prove_commute}
    \begin{split}
        W_i W_j e_k = \frac{\delta_{ik}\delta_{jk}}{4{\lambda_k}^2}e_k.
    \end{split}
\end{equation}
\eqref{eq:W_iW_je_k_to_prove_commute} is invariant under the interchange of $i$ and $j$.
Thus,  $W_i W_j e_k = W_j W_i e_k$ for any $i,j,k\in\mathbb{N}$.
Since $\{e_k\}_{k=1}^\infty$ is a complete orthonormal system,
it holds $W_i W_j = W_j W_i$ for any $i,j\in\mathbb{N}$,
meaning that Assumption \ref{ass:commute} holds.

\begin{example}
Consider the case where $\mathcal{H}=L^2(0,1)$, $A= \frac{\mathrm{d}^2}{\mathrm{d}x^2}$
with Dirichlet boundary conditions, and $e_k = \sqrt{2}\sin k\pi x$.
In this setting,
$A$ is self-adjoint, and $\{e_k\}_{k=1}^\infty$ forms a complete orthonormal system of eigenvectors of $A$.
Thus, Assumption \ref{ass:commute} holds.
\end{example}

\begin{lemma}\label{lem:simultaneously_diagonalized}
    Under Assumption \ref{ass:commute}, there exists a complete orthonormal system $\{z_k\}_{k=1}^\infty\subset\mathcal{H}$ and $\lambda_k^{(i)}\geq 0$ such that 
    \begin{equation}
    W_i z = \sum_{k=1}^\infty \lambda_{k}^{(i)} \langle z, z_k\rangle z_k \textrm{\quad for all $i\in\mathbb{N}$}.
\end{equation}
\end{lemma}

\begin{proof}
    Applying (\cite{samoilenko2012spectral} Chapter1.1, Theorem 1), with
    \begin{equation}\label{eq:simultaneous_diagonalization}
        \mathrm{d}E(\lambda_1,\lambda_2,\dots)
        = 
        \begin{cases}
            \prod_{i=1}^\infty P_i^{(\lambda_i)} & (\text{if $\lambda_i$ is an eigenvalue of $W_i$})\\
            0 & \text{(otherwise)}
        \end{cases},
    \end{equation}
    where $\mathrm{d}E(\lambda_1,\lambda_2,\dots)$ is the spectral measure in $\mathbb{R}^\infty$, and $P_i^{(\lambda_i)}$ is the orthogonal projection onto the eigenspace of $W_i$ with eigenvalue $\lambda_i$,
    we can choose an orthogonal basis of $\ran \prod_{i=1}^\infty P_i^{(\lambda_i)}$.
    Therefore, taking orthonormal basis from $\ran \prod_{i=1}^\infty P_i^{(\lambda_i)}$ for all pairs of eigenvalues, we obtain a complete orthogonal system of $\mathcal{H}$.
    Since this system consists of eigenvectors of $W_i$, the expression of equation \eqref{eq:simultaneous_diagonalization} holds, and $\lambda_k^{(i)}\geq 0$ follows since $W_i$ is positive semidefinite.
\end{proof}

If the controllability Gramian has more than $n$ nonzero eigenvalues, 
it is difficult to determine whether
the objective function of Problems \eqref{problem:VCS_infinite} or \eqref{problem:AECS_infinite} is convex.
In fact, under Assumption \ref{ass:commute}, $W(p)$ can be written as 
\begin{equation}
    W(p)z = \sum_{i\in I} p_i W_i z
    = \sum_{k=1}^\infty \sum_{i\in I} p_i\lambda_{k}^{(i)} \langle z,z_k\rangle z_k.
\end{equation}
This implies that eigenvalues of $W(p)$ are 
$\left\{\sum_{i\in I} p_i\lambda_{k}^{(i)}\right\}_{k=1}^\infty$
and corresponding eigenvectors are $\left\{z_k\right\}_{k=1}^\infty$.
If $W(p)$ has more than $n$ nonzero eigenvalues,
the objective function to Problems \eqref{problem:VCS_infinite} or \eqref{problem:AECS_infinite} can be rewritten as the form 
\begin{equation}\label{eq:minimum_of_objective_function_with_infinite_eigenvalues}
    \begin{split}
        \min \left\{h\left(\left.\sum_{i\in I} p_i\lambda_{k_1}^{(i)}, \sum_{i\in I} p_i\lambda_{k_2}^{(i)}
        ,\dots, \sum_{i\in I} p_i\lambda_{k_n}^{(i)}\right)\right\vert k_1,k_2,\dots,k_n \in \mathbb{N} \right\}
    \end{split}
\end{equation}
since $\mu_1(W(p))$, $\mu_2(W(p))$, $\dots$, $\mu_n(W(p))$ are $n$ larger eigenvalues of $W(p)$, and vary when $\{p_i\}_{i=1}^\infty$ changes.
The convexity of this reformulated objective function
is not immediately apparent,
because the minimum of convex functions is not necessarily convex.
For example, $\min\{(x-1)^2,(x+1)^2\}$ is not convex even though both $(x-1)^2$ and
$(x+1)^2$ are convex.
As a result, it is difficult to prove the uniqueness of the optimal solution to Problems \eqref{problem:VCS_infinite} and \eqref{problem:AECS_infinite}.

To ensure the convexity of the objective function of Problems \eqref{problem:VCS_infinite} or \eqref{problem:AECS_infinite}, we introduce an additional assumption.

\begin{assumption}\label{ass:n-spectrum}
    For fixed $I\subset \mathbb{N}$, there exists $\{z_{\sigma(k)}\}_{k=1}^n\subset \{z_k\}_{k=1}^\infty$
    such that the restriction of $W_i$ to $\left(\Span\{z_{\sigma(k)}\}_{k=1}^n\right)^\bot$
    is $0$.
\end{assumption}

Assumption \ref{ass:n-spectrum} means that the eigenvalues of $W(p)$
are $\left\{\sum_{i\in I}p_i \lambda_{\sigma(k)}^{(i)}\right\}_{k=1}^n$ and $0$.
This is because $W_i$ can be written as 
\begin{equation}
    W_i z = \sum_{k=1}^n \lambda_{\sigma(k)}^{(i)}\langle z,z_{\sigma(k)}\rangle z_{\sigma(k)},
\end{equation}
and thus
\begin{equation}
    W(p)z = \sum_{i=1}^n p_i W_i
    = \sum_{k=1}^n \left(\sum_{i=1}^n p_i \lambda_{\sigma(k)}^{(i)}\right)\langle z,z_{\sigma(k)}\rangle z_{\sigma(k)}.
\end{equation}
This implies that $k_1,k_2,\dots,k_n$ in \eqref{eq:minimum_of_objective_function_with_infinite_eigenvalues} are fixed and the convexity of the objective function $h$ is guaranteed.

We are now in a position to establish the uniqueness of VCS and AECS.
\begin{theorem}\label{thm:uniqueness_of_optimal_solution}
    Suppose that Assumptions \ref{ass:feasible_solution}, \ref{ass:input_node},  \ref{ass:commute} and \ref{ass:n-spectrum} hold.
    Then, the optimal solution to Problems \eqref{problem:VCS_infinite} and \eqref{problem:AECS_infinite} exists and is unique.
\end{theorem}
\begin{proof}
    We have already proven the existence in Theorem \ref{thm:existence_of_optimal_solution}.
    Thus, we now prove the uniqueness.
    Since the feasible region is convex,
    it suffices to show that the objective functions $f$ and $g$ are strictly convex.
    Without loss of generality, it is enough to consider the case $I=\{1,2,\dots,n\}$.
    
    The Hessian matrices of $f$ and $g$ are given by
    \begin{equation*}
        \left(\nabla^2 f(p)\right)_{\alpha \beta} = \sum_{k=1}^n\frac{\lambda_{\sigma(k)}^\alpha \lambda_{\sigma(k)}^\beta}{\left(\sum_{i=1}^n p_i \lambda_{\sigma(k)}^{(i)}\right)^2},\quad \left(\nabla^2 g(p)\right)_{\alpha \beta} = 2\sum_{k=1}^n\frac{\lambda_{\sigma(k)}^\alpha \lambda_{\sigma(k)}^\beta}{\left(\sum_{i=1}^n p_i \lambda_{\sigma(k)}^{(i)}\right)^3},
    \end{equation*}
    respectively. Thus, for $x\in\mathbb{R}^n$,
    \begin{equation}\label{eq:Hessian_product_f}
        x^\top \nabla^2 f(p)x = \sum_{k=1}^n \frac{1}{\left(\sum_{i=1}^n p_i\lambda_{\sigma(k)}^{(i)}\right)^2}(x^\top \lambda_{\sigma(k)})^2\geq0,
    \end{equation}
    \begin{equation}\label{eq:Hessian_product_g}
        x^\top \nabla^2 g(p)x = 2\sum_{k=1}^n \frac{1}{\left(\sum_{i=1}^n p_i\lambda_{\sigma(k)}^{(i)}\right)^3}(x^\top \lambda_{\sigma(k)})^2\geq0,
    \end{equation}
    which implies the convexity of $f$ and $g$, where
    $
        \lambda_{\sigma(k)} = 
        \begin{pmatrix}
            \lambda_{\sigma(k)}^{(1)}&
            \lambda_{\sigma(k)}^{(2)}&
            \cdots&
            \lambda_{\sigma(k)}^{(n)}
        \end{pmatrix}^\top.
    $
    
    Next, we show the strict convexity of $f$ and $g$.
    It is sufficient to show that if
    $x^\top\nabla^2 f(p)x = 0$ or $x^\top \nabla^2 g(p) x = 0$,
    then $x=0$. Note that if
    $x^\top\nabla^2 f(p)x = 0$ or $x^\top \nabla^2 g(p) x = 0$, then $x^\top \lambda_{\sigma(k)} = 0$ for $k=1,2,\dots,n$. This condition can be written in matrix form as
    \begin{equation}
    \begin{pmatrix}
        \lambda_{\sigma(1)}^{(1)} & \lambda_{\sigma(1)}^{(2)} & \cdots & \lambda_{\sigma(1)}^{(n)}\\
        \lambda_{\sigma(2)}^{(1)} & \lambda_{\sigma(2)}^{(2)} & \cdots & \lambda_{\sigma(2)}^{(n)}\\
        \vdots & \vdots & \ddots & \vdots\\
        \lambda_{\sigma(n)}^{(1)} & \lambda_{\sigma(n)}^{(2)} & \cdots & \lambda_{\sigma(n)}^{(n)}
    \end{pmatrix}
    \begin{pmatrix}
        x_1 \\ x_2 \\ \vdots \\ x_n
    \end{pmatrix}
    = 0.
    \end{equation}
    Therefore, to conclude that $x=0$, it is sufficient to show 
    $
    \begin{pmatrix}
        \lambda_{\sigma(1)}^{(i)}&
        \lambda_{\sigma(2)}^{(i)}&
        \cdots&
        \lambda_{\sigma(n)}^{(i)}&
    \end{pmatrix}^\top$
    $\textrm{$(i=1,2,\dots,n)$}
    $
    are linearly independent.
    
    Suppose
    \begin{equation}
        x_1
    \begin{pmatrix}
        \lambda_{\sigma(1)}^{(1)}\\
        \lambda_{\sigma(2)}^{(1)}\\
        \vdots\\
        \lambda_{\sigma(n)}^{(1)}\\
    \end{pmatrix}
    + x_2
    \begin{pmatrix}
        \lambda_{\sigma(1)}^{(2)}\\
        \lambda_{\sigma(2)}^{(2)}\\
        \vdots\\
        \lambda_{\sigma(n)}^{(2)}\\
    \end{pmatrix}
    +\cdots+x_n
    \begin{pmatrix}
        \lambda_{\sigma(1)}^{(n)}\\
        \lambda_{\sigma(2)}^{(n)}\\
        \vdots\\
        \lambda_{\sigma(n)}^{(n)}\\
    \end{pmatrix}
    = 0.
    \end{equation}
    Then, 
    $\sum_{i=1}^n x_i \lambda_{\sigma(k)}^{(i)}=0$ for $k=1,2,\dots,n$,
    and this implies that the controllability Gramian satisfies 
    \begin{equation}\label{eq:sum_of_component_Gramian_x_i}
        W(x)=\sum_{i=1}^n x_i W_i = 0.
    \end{equation}
    According to \cite[Theorem 6.5.3]{CurtainZwart2020}, it holds
    \begin{equation}\label{eq:Lyapunov_equation_projection_case}
         \langle x_iW_i z,A^* z\rangle + \langle A^*z,x_iW_iz\rangle = -x_i\langle P_i z,P_iz\rangle
    \end{equation}
    for $z\in \dom(A^*)$, where $P_i$ is the orthogonal projection defined by $P_i z = \langle z,e_i\rangle e_i$.
    Taking the sum over $i=1,2,\dots,n$ in \eqref{eq:Lyapunov_equation_projection_case} and using \eqref{eq:sum_of_component_Gramian_x_i}, we obtain
    \begin{equation}
        0 = -\sum_{i=1}^n x_i\langle P_iz, P_iz\rangle \quad \textrm{for $z\in\dom(A^*)$}.
    \end{equation}
    Moreover, since $A$ is an infinitesimal generator of the strongly continuous semigroup, $A$ is a densely defined closed operator \cite[Theorem 2.1.13]{CurtainZwart2020}, and thus, $\dom(A^*)$ is dense in $\mathcal{H}$ \cite[section VII.2, Theorem 3]{yosida2012functional}.
    Hence, for any $z\in \mathcal{H}$, there exists a sequence $\{z_j\}_{j=1}^\infty\subset\dom(A)$ such that $z_j \rightarrow z$.
    Passing to the limit, we deduce that
    \begin{equation}
        \begin{split}
            \sum_{i=1}^n x_i \langle P_iz,P_iz\rangle
            &= \lim_{j\to\infty}\sum_{i=1}^\infty x_i \langle P_iz_j,P_iz_j\rangle
            = 0.
        \end{split}
    \end{equation}
    Thus,
    \begin{equation}\label{eq:sum_projection_e_i}
        \sum_{i=1}^n x_i \langle P_iz,P_iz\rangle = 0 \quad \text{for any $z\in\mathcal{H}$}.
    \end{equation}
    By choosing $z = e_j$ in \eqref{eq:sum_projection_e_i} and noting that $\langle P_i e_j,P_ie_j\rangle = \delta_{ij}$, where $\delta_{ij}$ is Kronecker delta, we obtain $x_j = 0$ for $j=1,2,\dots,n$, which in turn establishes the strict convexity of the objective functions $f$ and $g$.
\end{proof}

Theorem \ref{thm:uniqueness_of_optimal_solution} shows that
VCS and AECS 
 are unique, which supports their use as centrality measures.

We note that Assumptions \ref{ass:feasible_solution}, \ref{ass:input_node},  \ref{ass:commute} and \ref{ass:n-spectrum} do not mean the controllability of the system.

\begin{definition}[Approximately controllable, {\cite[Definition 6.2.18]{CurtainZwart2020}}]

The reachability subspace of system \eqref{eq:state_equation_infinite_with_input_diagonal_B} is defined as
\begin{equation}
    \begin{split}
        \mathcal{R} = \{ z\in\mathcal{H} &\mid \textrm{there exist $T>0$ and $u\in L^2([0,T];H)$}\\ 
        &\textrm{such that $z = \int_0^T \exp((T-t)A)Bu(t)$}\mathrm{d}t \}.
    \end{split}
\end{equation}
System \eqref{eq:state_equation_infinite_with_input_diagonal_B}
is called approximately controllable in infinite time if $\overline{\mathcal{R}} = \mathcal{H}$.
\end{definition}

In the finite-dimensional case, it was assumed that  system \eqref{eq:input_network} is controllable. 
However, in the infinite-dimensional case, under Assumptions \ref{ass:feasible_solution}, \ref{ass:input_node}, \ref{ass:commute}, and \ref{ass:n-spectrum}, we cannot determine that  system \eqref{eq:state_equation_infinite_with_input_diagonal_B} is approximately controllable or not.
In fact, according to \cite[Theorem 6.4.5]{CurtainZwart2020}, system \eqref{eq:state_equation_infinite_with_input_diagonal_B} is approximately controllable in infinite time if and only if
\begin{equation}\label{eq:equivalent_form_approximately_controllable}
        B^*\exp(tA)^* z = 0 \textrm{ for all $t\geq 0$} \Rightarrow z = 0.
\end{equation}
This condition, however, is not necessarily satisfied in the case considered in this paper.
For example, if $A$ is self-adjoint and $\{e_i\}_{i=1}^\infty$ are eigenvectors of $A$,
Assumptions \ref{ass:feasible_solution}, \ref{ass:input_node}, \ref{ass:commute}, and \ref{ass:n-spectrum} hold, but condition \eqref{eq:equivalent_form_approximately_controllable} does not hold for $e_i$ with $i\notin I$ because $B^*\exp(tA)^*e_i = e^{\lambda_i t}Be_i = 0$, but $e_i\neq 0$.

\subsection{VCS for a special case}

In this section, we assume Assumptions \ref{ass:feasible_solution}, \ref{ass:input_node},
\ref{ass:commute}, and \ref{ass:n-spectrum}
hold, and use same notation as in Section \ref{sec:definition_of_controllability_score}.

We assume $A$ is self-adjoint and $\{e_i\}_{i=1}^\infty$ is a complete orthonormal system which satisfies $Ae_i = \lambda_i e_i$ for all $i\in\mathbb{N}$.
Then, $\lambda_i\in\mathbb{R}$ and 
$\exp(tA)$ has the form 
\begin{equation}
    \exp(tA)z = \sum_{i=1}^\infty e^{\lambda_i t} \langle z,e_i\rangle e_i.
\end{equation}
Since $A$ is exponentially stable, $\lambda_i < 0$ for all $i\in\mathbb{N}$.
Then, the controllability Gramian has the following form.
\begin{equation}
    \begin{split}
        W(p)z 
        &= \int_0^\infty \exp(tA)B B^*\exp(tA)^* z\mathrm{d}t \\
        &= \int_0^\infty \sum_{i\in I} p_i e^{2\lambda_i t}\langle z,e_i\rangle e_i \mathrm{d}t
        = \sum_{i\in I} \left(\frac{p_i}{-2\lambda_i}\right) \langle z,e_i\rangle e_i
    \end{split}
\end{equation}
Thus, $n$ larger eigenvalues of $W(p)$ are 
$\left\{\frac{p_i}{-2\lambda_i}\right\}_{i\in I}$, and the objective function of VCS is 
\begin{equation}\label{eq:VCS_symmetry}
    \begin{split}
        f(p) 
        = \prod_{i\in I}\log \left(\frac{p_i}{-2\lambda_i}\right)
        = \log \prod_{i\in I} p_i - \log \prod_{i\in I}(-2\lambda_i).
    \end{split}
\end{equation}
Equation \eqref{eq:VCS_symmetry} implies that these eigenvalues have the equivalent effect to the objective 
function and thus VCS is $\frac{1}{n}$ for all $i\in I$.

Therefore, this special case can be considered an extension to
the infinite-dimensional setting,
since in the finite -dimensional setting,
VCS takes the same value for all nodes
if $A$ is symmetric \cite{SatoKawamura2024}.

\subsection{Remark}

In Section \ref{sec:definition_of_controllability_score},
we extended the definition of the controllability scores to infinite-dimensional systems.
In Section \ref{sec:existence_of_VCS_and_AECS}, we proved the existence, and in Section \ref{sec:uniqueness_of_VCS_and_AECS}, we established their uniqueness.
In these sections, the state space $\mathcal{H}$ was assumed to be infinite dimensional, but the arguments in these sections
remain valid even if $\mathcal{H}$ is finite dimensional.
Moreover, when $\dim \mathcal{H} = n$ and the number of evaluating state nodes is $n$, the result on the existence and uniqueness of the controllability scores in Sections \ref{sec:existence_of_VCS_and_AECS} and \ref{sec:uniqueness_of_VCS_and_AECS} is consistent with the findings in \cite{satoterasaki2024}.
Moreover, we proved the existence of the controllability score in infinite-dimensional systems under Assumption \ref{ass:p_is_less_than_a}---which is always satisfied in the finite-dimensional cases---and Assumption \ref{ass:feasible_solution}, which is also assumed in the finite-dimensional setting \cite{satoterasaki2024}.
However, there is a difference on the uniqueness result between our work in Section \ref{sec:uniqueness_of_VCS_and_AECS} in this paper
and the results in \cite{satoterasaki2024}.
In fact, Assumptions \ref{ass:input_node} and \ref{ass:n-spectrum} in Section \ref{sec:uniqueness_of_VCS_and_AECS} are always satisfied when $\dim \mathcal{H}<\infty$.
However, Assumption \ref{ass:commute} in Section \ref{sec:uniqueness_of_VCS_and_AECS}---which plays a role in the infinite-dimensional case---is not imposed  in the finite-dimensional case \cite{satoterasaki2024}.
Moreover, the stability of $A$ is one of the cases discussed in \cite{satoterasaki2024}.
Therefore, the results in Section \ref{sec:uniqueness_of_VCS_and_AECS}
can be regarded as a partial extension of the existence results for
the controllability scores to the infinite-dimensional setting.


\section{Numerical Experiments}\label{sec:numerical_experiments}
In this section, we provide an application of the controllability scores to an infinite-dimensional system.
Before presenting examples, we emphasize an important caution.
In the previous section, we showed the uniqueness of the scores under Assumptions \ref{ass:feasible_solution}, \ref{ass:input_node},  \ref{ass:commute} and \ref{ass:n-spectrum}.
However, especially Assumptions \ref{ass:commute} and \ref{ass:n-spectrum} are strong. 
These two assumptions are closely related to the selection of state nodes.
If the state nodes to be evaluated are changed,
these assumptions may no longer hold, even if the other state nodes satisfy them.
Thus, ensuring the uniqueness of the controllability scores
requires careful selection not only of the structure $A$,
but also of the state nodes$\{e_i\}$.
For example, if $A$ can be expressed as $A = \sum a_i \langle \cdot, e_i\rangle e_i$
where $\{e_i\}$ are orthonormal eigenvectors of $A$,
choosing $\{e_i\}$ as the state nodes guarantees that
Assumptions \ref{ass:commute} and \ref{ass:n-spectrum} hold.
However, cases where $A$ and $\{e_i\}$ satisfy Assumptions \ref{ass:commute} and \ref{ass:n-spectrum} do not occur often.
Moreover, even if Assumptions \ref{ass:commute} and \ref{ass:n-spectrum} hold, the state nodes $\{e_i\}$ may not have physical relevance or importance from a system control perspective.
The following example illustrates such a case, where Assumptions \ref{ass:feasible_solution}, \ref{ass:input_node},  \ref{ass:commute} and \ref{ass:n-spectrum} hold but assigning physical meaning may be challenging.

Now, we consider the following heat equation with Dirichlet boundary conditions.
\begin{equation}\label{eq:heat_equation_Dirichlet}
    \begin{split}
        &\frac{\partial u}{\partial t} = \frac{\partial^2 u}{\partial x^2}, \quad \textrm{$t>0$, $x\in(0,1)$},\\
        &u(t,0) = u(t,1) = 0, \quad \textrm{$t>0$}.
    \end{split}
\end{equation}
This can be interpreted as a metal rod with a heat bath at both ends 
and a fixed temperature of 0 at each end.
Here, $u(t,x)$ represents the temperature at position $x$ and time $t$.
Thus, controlling the system means controlling the temperature distribution
along the rod.

To formulate this as an equation of the form $\dot{u}(t) = Au(t)$ in $L^2(0,1)$,
we introduce a linear operator $A$ defined as 
\begin{equation}\label{eq:definition_Dirichlet_Lapracian}
        Au = \frac{\mathrm{d}^2u}{\mathrm{d}x^2}, \quad \textrm{$u\in \dom(A) = H_0^1(0,1) \cap H^2(0,1)$},
\end{equation}
where $H^2(0,1)$ and $H_0^1(0,1)$ are Sovolev spaces.
Thus, equation (\ref{eq:heat_equation_Dirichlet}) can be expressed as
\begin{equation}
    \dot{u}(t) = Au(t), \quad \textrm{$u\in\dom(A)$}.
\end{equation}
Then, $A$ generates a semigroup expressed as
\begin{equation}\label{eq:spectral_decomposition_of_exp_tA}
    \exp(tA)u = \sum_{k=1}^\infty e^{-\pi^2 k^2 t}\langle u,e_k\rangle e_k,
\end{equation}
where $e_k = \sqrt{2}\sin(k\pi x)$.
This is because $e_k\in \dom(A)$, $Ae_k = -\pi^2k^2 e_k$
and $\{e_k\}_{k=1}^\infty$ forms a complete orthonormal system in $L^2(0,1)$.
Therefore, we choose controlling nodes as $\{e_k\}_{k=1}^\infty$.
This is an example in which uniqueness of the scores is guaranteed.

We first define $P_k$ as $P_k u = \langle u,e_k\rangle e_k$.
To calculate the controllability Gramian, we compute
$W_k = \int_0^\infty \exp(tA) P_k P_k^* \exp(tA)^*\mathrm{d}t$.
In this case, $\exp(tA)^* = \exp(tA)$, because $A$ is self-adjoint.
Thus, $\exp(tA)P_kP_k^* \exp(tA)^* e_l = e^{-2\pi^2 l^2 t}\delta_{kl}e_l$ and 
\begin{equation}
    \begin{split}
        W_k e_l
        = \int_0^\infty \exp(tA)P_k P_k^* \exp(tA)^* e_l \mathrm{d}t
        = \int_0^\infty e^{-2\pi^2 l^2 t}\delta_{kl}e_l \mathrm{d}t
        = \frac{\delta_{kl}}{2\pi^2 l^2}e_k.
    \end{split}
\end{equation}
Here, $\delta_{kl}$ is Kronecker delta.
This implies that
\begin{equation}
    W_k u = \frac{1}{2\pi^2k^2}\langle u,e_k\rangle e_k.
\end{equation}
Since $\{e_k\}_{k=1}^\infty$ is orthonormal, we have
\begin{equation}
        W_k W_l u
        = W_k\left(\frac{1}{2\pi^2l^2}\langle u,e_l\rangle e_l\right)
        = \frac{1}{4\pi^4 k^2l^2}\langle u,e_l\rangle \langle e_l,e_k\rangle e_k
        = \frac{\delta_{kl}}{4\pi^4 k^2l^2} \langle u,e_l\rangle e_k.
\end{equation}
This implies that $W_kW_l = W_l W_k$,
so the set $\{W_k\}_{k=1}^\infty$ is commutative
and Assumption \ref{ass:commute} holds.
Assumption \ref{ass:n-spectrum} is also satisfied
because the restriction of $W_k$ to $\left(\Span\{e_k\}\right)^\bot$ is zero.
Therefore, the restriction of $\{W_k\}_{k\in I}$ to $\left(\Span\{e_k\}_{k\in I}\right)^\bot$ is also zero.
Moreover, it is straightforward to verify that
the other assumptions are satisfied as well.
Thus, the uniqueness of controllability score is guaranteed in this case.

If we want to evaluate nodes corresponding to 
a set $I\subset \mathbb{N}$, which contains $n$ elements,
then controllability Gramian $W(p) = \sum_{k\in I} p_k W_k$ satisfies
\begin{equation}
    W(p)u = \sum_{k\in I} \frac{p_k}{2k^2\pi^2}\langle u,e_k\rangle e_k.
\end{equation}

From the form, $n$ largest eigenvalues of $W(p)$
with multiplicity are given by
$\left\{\frac{p_k}{2k^2\pi^2}\right\}_{k\in I}$.
Therefore, the objective function of VCS is 
\begin{equation}\label{eq:objective_function_VCS_sin}
    f(p) = -\log\left(\prod_{k\in I} \frac{p_k}{2\pi^2k^2}\right)
    =-\log\left(\prod_{k\in I} p_k\right) + \log\left(\prod_{k\in I}2\pi^2k^2\right),
\end{equation}
and the objective function of AECS is 
\begin{equation}\label{eq:objective_function_AECS_sin}
    g(p) = \sum_{k\in I}\frac{2\pi^2k^2}{p_k}.
\end{equation}
Equation \eqref{eq:objective_function_VCS_sin} implies that all $\{p_k\}_{k\in I}$ have the same influence on $f(p)$,
and $f(p)$ achieves its maximum value when $p_k = \frac{1}{n}$ for all $k\in I$.
In this case, VCS is same for all nodes.
On the other hand, AECS indicates that each node has a different importance.
Table (\ref{table:heat_eq_AECS_sin_result}) presents the result.
AECS gives more weight to the impact of large $k$ values.

\begin{table}[H]
    \centering
    \caption{AECS of $e_k = \sin k\pi x$. Evaluate $e_k$ for $k\in I$. 
    The position of the left nodes and the right scores correspond}\label{table:heat_eq_AECS_sin_result}
    \begin{tabular}{c|c}
        $I$ & AECS\\\hline
        $\{1, 2, 3, 4\}$ & $(0.10, 0.20, 0.30, 0.40)$\\
        $\{1, 2, 3, 5\}$ & $(0.09, 0.18, 0.27, 0.45)$\\
        $\{1, 2, 3, 6\}$ & $(0.08, 0.16, 0.25, 0.50)$\\
        $\{2, 3, 4, 5\}$ & $(0.14, 0.21, 0.28, 0.35)$\\
        $\{2, 3, 4, 6\}$ & $(0.13, 0.20, 0.26, 0.40)$\\
        $\{3, 4, 5, 6\}$ & $(0.16, 0.22, 0.27, 0.33)$\\
    \end{tabular}
\end{table}

In the above experiment,
$\sin(k\pi x)$ represents a frequency component that is difficult to handle in practical control applications.
Moreover, in general, frequency components with large $k$
are considered noise. 
However, while VCS treats all components as equivalent,
 AECS assigns higher importance to components with larger $k$.
This discrepancy does not align with our intuition and makes it challenging to interpret the scores meaningfully.

\section{Concluding remarks}
\label{sec:conclusion}

In this research, we extended the definitions of the controllability scores---namely, the Volumetric Controllability Score (VCS) and the Average Energy Controllability Score (AECS)---to infinite-dimensional systems and proved their existence under weak assumptions, and uniqueness under certain assumptions. Moreover, we demonstrated that VCS produces a distinctive outcome when the operator $A$
is self-adjoint, possesses a complete orthonormal system of eigenvectors, and the controlling nodes are chosen from this system.
We also conducted a numerical experiment applying the controllability score to the heat equation on a unit interval with Dirichlet boundary conditions. In this experiment, all the necessary assumptions are satisfied, ensuring uniqueness; however, interpreting the result in a meaningful way remains challenging.

To prove the uniqueness of the scores, we had to impose strong assumptions that, among other effects, restrict the number of nonzero eigenvalues. In cases where the spectrum is infinite, the objective functions are defined as the minimum over a family of convex functions, and the overall convexity of these functions is not immediately apparent. Since these strong assumptions are restrictive, only a limited class of systems satisfy them. Although we were able to establish the existence of an optimal solution without these stringent assumptions, the uniqueness of this solution under weaker or no assumptions remains unproven. In other words, demonstrating the uniqueness of the optimal solution without relying on these restrictive conditions is an open problem---a challenge that we intend to address in future work.

\bibliographystyle{siamplain}
\bibliography{references}

\appendix

\section{Proof of Theorem \ref{thm:existence_of_optimal_solution}}\label{sec:proof_existence_optimal_solution}

In this section, we prove Theorem \ref{thm:existence_of_optimal_solution}.

Define $S$ as 
\begin{equation}
    S \coloneq  \left\{ \{b_i\}_{i=1}^\infty \left\vert 0\leq b_i\leq a_i, \sum_{i=1}^\infty b_i = 1\right.\right\},
\end{equation}
where $\{a_i\}_{i=1}^\infty$ is a sequence defined in Assumption \ref{ass:p_is_less_than_a}.
Then, the following lemma holds.

\begin{lemma}\label{lem:compactness_of_l1_subset}
$S$ is a compact subset of $\ell^1(\mathbb{N})$.
\end{lemma}
\begin{proof}
    Let $\{\beta^{(n)}\}_{n=1}^\infty$ be a sequence in $S$, where $\beta^{(n)} \coloneq \{b_i^{(n)}\}_{i=1}^\infty \in S$.
    Since $0\leq b^{(n)}_1\leq a_1$ holds for any $n\in\mathbb{N}$, the sequence $\{b_1^{(n)}\}_{n=1}^\infty$ has a convergent subsequence $\{b^{(1_n)}_1\}_{n=1}^\infty$.
    Similary, since $0\leq b^{(1_n)}_2 \leq a_2$, $\{b^{(1_n)}_2\}_{n=1}^\infty$ also has
    a convergent subsequence $\{b^{(2_n)}_2\}_{n=1}^\infty$.
    Repeating this procedure for each coordinate and applying the diagonal argument,
    we obtain a subsequence $\{\beta^{(n_n)}\}_{n=1}^\infty$.
    For all $i\in\mathbb{N}$, the sequence $\{b^{(n_n)}_i\}_{n=1}^\infty$ converges,
    and $0\leq b^{(n_n)}_i\leq a(i)$ holds.
    Define $c_i\coloneq \lim_{n\to\infty}b^{(n_n)}_i$.
    Since $0\leq b^{(n_n)}_i\leq a_i$, it follows that
    $|b^{(n_n)}_i - c_i|\leq a_i$.
    Moreover, since $\sum_{i=1}^\infty a_i <\infty$,
    we can apply Lebesgue convergence theorem.
    This yields
    $\lim_{n\to\infty} \sum_{i=1}^\infty |b^{(n_n)}_i - c_i| = 0$,
    which implies that the subsequence $\{\beta^{(n_n)}\}_{n=1}^\infty$ converges to $c$ with respect to $\ell^1$ norm.
    Furthermore, it is straightforward to verify that
    $\{c_i\}_{i=1}^\infty \in S$.
    Thus, $\{\beta^{(n)}\}_{n=1}^\infty$ has a subsequence that converges in $S$. 
    Therefore, $S$ is compact.
\end{proof}

\begin{lemma}
    Under Assumption \ref{ass:p_is_less_than_a}, the feasible region $\mathcal{F}$ is defined as 
    \begin{equation}
        \mathcal{F} \coloneq S \cap X,
    \end{equation}
    where
    \begin{equation}
        X = \{\{p_i\}_{i=1}^\infty \mid \mu_n(W(p))>0\}.
    \end{equation}
    Then $\mathcal{F}$ is a convex set.
\end{lemma}
\begin{proof}
    Since the convexity of $S$ is straightforward, we show the convexity of $X$.
    Then, it suffices to show that $\mu_n(W((1-\theta)p + \theta q))>0$ for any $\theta\in(0,1)$ and $p,q\in X$.
    Using the max-min Theorem (\cite{lax2014functional}, Chapter 28, Theorem 4), we obtain a $n$-dimensional subspace $M$ which satisfies
    \begin{equation}
        \mu_n(W(p)) = \min_{x\in M, \|x\| = 1}\langle W(p)x,x\rangle.
    \end{equation}
    Thus, 
    \begin{equation}
        \begin{split}
            \mu_n(W((1-\theta)p + \theta q))
            &\geq \min_{x\in M,\|x\|=1} \langle W((1-\theta)p+\theta q)x,x\rangle\\
            &\geq (1-\theta)\min_{x\in M,\|x\|=1} \langle W(p)x,x\rangle + \theta\min_{x\in M,\|x\|=1} \langle W(q)x,x\rangle\\
            &\geq (1-\theta)\mu_n(W(p))
            >0,
        \end{split}
    \end{equation}
    since $W(q)$ is positive semidefinite and thus $\langle W(q)x,x\rangle\geq 0$.
    This implies $(1-\theta)p+\theta q\in X$.
\end{proof}

In the proof of Theorem \ref{thm:existence_of_optimal_solution}, we use the notation of $B(p)$, where $p = \{p_i\}_{i=1}^\infty$ and $B(p)$ is a linear operator defined as $B(p)e_i = \sqrt{p_i}e_i$.

\begin{proof}[Proof of Theorem \ref{thm:existence_of_optimal_solution}]
    We first show $f$ and $g$ are continuous on $\mathcal{F}$.
    It suffices to demonstrate $\mu_k(W(\cdot )): \mathcal{F}\to \mathbb{R}$ is continuous.
    Let $\{p^{(m)}\}_{m=1}^\infty \subset \mathcal{F}$ be a sequence that converges to $p^{(\infty)}\in \mathcal{F}$
    with respect to $\ell^1$ norm.
    Then, $B(p^{(m)})$ converges to $B(p^{(\infty)})$ with respect to the operator norm.
    In fact, since $B(p)$ is self-adjoint, $\|B(p)\|$ equals to the largest spectrum of $B(p)$, which is $\max_i \sqrt{p_i}$.
    Moreover, since $\max_i |p(i)| \leq \sum_{i=1}^\infty |p(i)|$,
    the convergence of $p^{(m)}$ to $p^{(\infty)}$ with respect to the $\ell^1$ norm
    implies that $B(p^{(m)})$ converges to $B(p^{(\infty)})$.
    In addition, $W(p^{(m)})$ converges to $W(p^{(\infty)})$.
    This is because $\|B(p^{(m)})\| < 2\|B(p^{(\infty)})\|$ holds for sufficiently large $m$, and it holds
    \begin{equation}
        \|\exp(tA)B(p^{(m)})B(p^{(m)})^*\exp(tA)^*\|\leq 4\|B(p^{(\infty)})\|^2M^2\exp(-2\omega t).
    \end{equation}
    Therefore, applying Lebesgue convergence theorem, it can be shown that 
    \begin{equation}\label{eq:W_m_converges_to_W_infty_with_respect_to_operator_norm}
        \begin{split}
            W(p^{(m)}) &= \int_0^\infty \exp(tA)B(p^{(m)})B(p^{(m)})^*\exp(tA)^* \mathrm{d}t
            \\
            &\xrightarrow[m\to\infty]{}
            \int_0^\infty \exp(tA)B(p^{(\infty)})B(p^{(\infty)})^*\exp(tA)^* \mathrm{d}t
            = W(p^{(\infty)}).
        \end{split}
    \end{equation}
    According to \cite[Lemma 1]{Yamamoto1968}, 
    $|\mu_k(W(p^{(m)})) - \mu_k(W(p^{(\infty)}))|\leq \|W(p^{(m)}) - W(p^{(\infty)})\|$, and equation \eqref{eq:W_m_converges_to_W_infty_with_respect_to_operator_norm} implies 
    $\mu_k(W(p^{(m)}))$ converges to $\mu_k(W(p^{(\infty)}))$ as $m\to\infty$.
    This implies that $\mu_k(W(\cdot))$ is continuous.
    Therefore, $f$ and $g$ are also continuous on $\mathcal{F}$.
    
    Then, we show the existence of an optimal solution.
    Let $\mathcal{F}_0$ be 
    \begin{equation}
        \mathcal{F}_0 = \{ p\in\mathcal{F}\mid h(p)\leq h(p_0)\},
    \end{equation}
    where $h$ represents $f$ or $g$ of Problem \eqref{problem:VCS_infinite} or \eqref{problem:AECS_infinite}.
    Since $h$ is continuous, it is sufficient to show $\mathcal{F}_0$ is compact.
    Moreover, since $\mathcal{F}_0 \subset \mathcal{F} = S\cap X \subset S$ and $S$ is compact, it is sufficient to show $\mathcal{F}_0$ is closed with respect to $\ell^1$ norm.
    Let $\{p^{(m)}\}_{m=1}^\infty\subset \mathcal{F}_0$ converges to $p^{(\infty)}$ with respect to $\ell^1$ norm.
    Then, if $p^{(\infty)}\notin X$, $\mu_n(W(p^{(m)}))$ converges to $0$.
    This implies $h(p^{(m)})>h(p_0)$ for sufficiently large $m$.
    This is contradiction.
    Thus, $p^{(\infty)}\in X$.
    Then, we can define $h(p^{(\infty)})$, and
    \begin{equation}
        h(p^{(\infty)}) = h(\lim_{m\to\infty}p^{(m)}) = \lim_{m\to\infty}h(p^{(m)}) \leq h(p_0),
    \end{equation}
    because $h$ is continuous.
    This implies $p^{(\infty)}\in\mathcal{F}_0$ and $\mathcal{F}_0$ is closed.
    Therefore, by compactness of $\mathcal{F}_0$ and continuity of $h$, an optimal solution exists.
\end{proof}

\end{document}